\documentclass[11.95pt,reqno]{amsart}
\usepackage{graphicx}
\usepackage{psfrag}
\usepackage{pxfonts}
\usepackage{mathrsfs}
\usepackage{color}
\usepackage{amsmath,amssymb}

\numberwithin{equation}{section}

\theoremstyle{plain}
\newtheorem{maintheorem}{Theorem}

\newtheorem{theorem}{Theorem}[section]

\newtheorem{lemma}[theorem]{Lemma}
\newtheorem{definition}[theorem]{Definition}

\newtheorem{question}{Question}
\theoremstyle{remark}

\begin{document}
\title[A relation between entropy and transitivity of Anosov diffeomorphisms]
{A relation between entropy and transitivity of Anosov diffeomorphisms}
\author[F. Micena]{Fernando Pereira Micena}
\address{Instituto de Matem\'{a}tica e Computa\c{c}\~{a}o,
  IMC-UNIFEI, Itajub\'{a}-MG, Brazil.}
\email{fpmicena82@unifei.edu.br}


\baselineskip=18pt              


\begin{abstract}
It is known that transitive Anosov diffeomorphisms have a unique measure of maximal entropy (MME). Here we discuss the converse question. Under suitable hypothesis on Lyapunov exponents on the set of periodic points and the structure of the MME we get transitivity of  $C^1-$Anosov diffeomorphisms.
\end{abstract}

\subjclass[2010]{}

\maketitle


\section{Introduction}\label{sec:intro}

Let $(M,g)$ be a $C^{\infty}$ compact, connected and boundaryless Riemannian manifold and $f:M \rightarrow M$ be a $C^1-$diffeomorphism. We say that $f$ is an Anosov diffeomorphim if there are numbers $0 < \beta < 1 < \eta, C > 0$ and a $Df-$invariant continuous splitting $T_xM = E^u_f(x) \oplus E^s_f(x),$ such that for any $n \geq 0,$
$$ ||Df^n(x) \cdot v || \geq \frac{1}{C} \eta^n ||v||, \forall v \in E^u_f(x), \quad ||Df^n(x) \cdot v || \leq C \beta^n ||v||, \forall v \in E^s_f(x).$$

It is known that the bundles $E^s_f, E^u_f$ are  respectively integrable to invariant  stable and unstable foliations denoted $\mathcal{F}^s_f$ and $\mathcal{F}^u_f.$  Given a point $x \in M,$ the stable and unstable leaves that contain $x$ are respectively characterized by
$$W^s_f(x) = \{ y \in M| \; \displaystyle\lim_{n \rightarrow +\infty}d(f^n(x) , f^n(y)) = 0  \},$$
$$W^u_f(x) = \{ y \in M| \; \displaystyle\lim_{n \rightarrow +\infty}d(f^{-n}(x) , f^{-n}(y)) = 0  \}.$$
If $f$ is $C^r, r\geq 1,$ then $W^{\ast}_f(x), \ast \in \{s,u\}$ are embedded $C^r$ submanifolds of $M.$

From now on we denote by $m$ the probability Lebesgue measure on $M$ induced by $g,$ and $Per(f)$ the set of periodic points of a diffeomorphism $f .$

Anosov diffeomorphims play an important role in the theory of dynamical systems and this class of diffeomorphisms satisfies many rich dynamical properties, as shadowing and closing lemmas and, in the case that the Anosov diffeomorphism is $C^2$ and preserves a measure $\mu $ absolutely continuous, this measure is ergodic.  A step towards  classifying Anosov diffeomorphisms, up to topological conjugacy, would consist in showing that every Anosov diffeomorphism is transitive.

\begin{definition}
Let $(X,d)$ be a compact metric space and $f: X \rightarrow X$ a continuous function. We say that $f$ is transitive if for any nonempty open sets $U$ and $V$ there exists an integer $N$ such that $f^{-N}(V) \cap U \neq \emptyset,$ or equivalently, there exists a point $x \in X$ with dense orbit.
\end{definition}

There are examples of non transitive Anosov flows, see  \cite{FW}, but there are no known examples of non transitive Anosov diffeomorphisms.

For Anosov diffeomorphisms, since they are Axiom A diffeomorphisms, transitivity property is equivalent to $\Omega(f) = M,$ where $\Omega(f)$ is the non wandering set of $f.$
Under transitivity assumption, Anosov diffeomorphisms have a unique measure of maximal entropy, see \cite{Bowen}, for instance. We could ask about the converse.

\begin{question}\label{quest}
Given $f: M \rightarrow M$ an Anosov diffeomorphism having unique measure of maximal entropy, is $f$ transitive ?
\end{question}

In general this question is totally inconclusive. Note that, in case  of a possible non transitive Anosov diffeomorphism,   could occur  $\Omega(f) = \Omega_1 \cup \ldots \cup \Omega_s, s > 1,$ a union of distinct basic sets with mutually different topological entropies, so $f$ could have a unique measure of maximal entropy but would not be transitive.  We could ask under what conditions a Anosov diffeomorphism having a unique measure of maximal entropy is transitive.

In this work we present a partial answer to the above Question \ref{quest} connecting transitivity, volume growth of unstable leaves, Lyapunov exponents and some regularity of the measure of maximal entropy.

\begin{maintheorem}\label{t1}
If $f: M \rightarrow M$ is a $C^1-$Anosov diffeomorphism such that $p \mapsto \Lambda^u_f(p)$ is constant on $Per(f),$ then $h_{top}(f) = \Lambda^u_f,$ where $\Lambda^u_f$ here denotes the common value $\Lambda^u_f(p), p \in Per(f).$ Additionally, suppose that  $f$ has
a unique measure of maximal entropy, $\mu.$ If $\mu$ is an absolutely continuous measure with continuous density with respect to $m,$  then  $f$ is  transitive.
\end{maintheorem}

In Theorem \ref{t1}, the value $\Lambda^u_f(p)$ denotes the sum of all positive Lyapunov exponents at the periodic point $p.$ Analogously, $\Lambda^s_f(p),$ denotes the sum of all negative Lyapunov exponents at the periodic point $p.$

\section{Comments and Applications}

In Theorem \ref{t1} we are not presuming that the density $ \rho :=\frac{d\mu}{dm}$ is positive. If that were the case, the result would be trivial, since $\mu$ is ergodic.

We observe that when the regularity of an Anosov diffeomorphism $f$ is only $C^1,$ such diffeomorphism not necessarily satisfies the absolute continuity of the unstable and stable foliations, see \cite{RY}. The absolute continuity of the unstable and stable foliations ever holds in $C^2$ context, it is the main ingredient in the proof of ergodicity of $C^2$ conservative Anosov diffeomorphisms.  In $C^2$ setting if $f$ is an Anosov diffeomorphism   preserving an absolute continuous measure with respect to $m,$ we can conclude $\Omega(f) = M,$ see \cite{BR}.  In contrast with this, for $C^1$ Anosov diffeomorphisms, Bowen in \cite{BoInv} constructed an example of Axiom $A$ diffeomorphism with a basic set $\Omega_s,$ with empty interior and such that $m(\Omega_s) > 0.$

Theorem \ref{t1} allows us to study the topological entropy for a class of $C^1$ pertur-\\bations of Anosov diffeomorphisms that act as identity on a invariant direction. These perturbations have special importance in the problems related to increasing or decreasing of center Lyapunov exponents in partially hyperbolic context.

Consider $A: \mathbb{T}^3  \rightarrow  \mathbb{T}^3  $ a linear Anosov automorphism such that $T\mathbb{T}^3 = E_A^{ss} \oplus E_A^{s} \oplus E_A^u $ and the corresponding eigenvalues $\lambda_A^{ss}, \lambda_A^s, \lambda_A^u$ satisfy $\lambda_A^{ss} < \lambda_A^s < 0 < \lambda_A^u.$ The stable bundle of $A$ is the sum $E_A^{ss} \oplus E_A^{s}.$

In the remarkable work of Baraviera and Bonatti \cite{BB}, with the goal to perturb center Lyapunov exponents, they construct smooth perturbations $\psi: \mathbb{T}^3 \rightarrow \mathbb{T}^3 ,$ which are sufficiently $C^1-$close to the identity such that \begin{itemize}
\item $\psi$ are volume preserving,
\item $\psi$ leave invariant the direction $E^u_A,$
\item $\psi$ act on the direction $E^u_A$ as the identity.
\end{itemize}

Moreover, by the construction in \cite{BB}, the map $f = A \circ \psi: \mathbb{T}^3 \rightarrow \mathbb{T}^3  $ is a smooth and volume preserving Anosov diffeomorphism such that $$ \lambda^{s}_f =  \int_{\mathbb{T}^3} \log(||Df(x)|E^s_f||)dm > \lambda_A^s. $$ Since $\psi$ are smooth perturbations suffciently $C^1-$close to the identity, $A$ and $f$ are conjugated by a conjugacy $h$ and consequently have the same topological entropy. Moreover, $\psi $ act on the direction $E^u_A$ as the identity, then
 $$ \lambda^{u}_f =  \int_{\mathbb{T}^3} \log(||Df(x)|E^u_f||)dm = \lambda^u_A. $$ By applying the Pesin formula
 $$\lambda^u_A = h_m(f),  $$
the volume $m$ is also the measure of maximal entropy for $f.$


Let us we consider a $C^1$ perturbation  $\psi,$ non necessarily $C^1$ close to the identity, requiring only that  $\psi$ acts  as the identity on the direction $E^u_A.$ If $f = A \circ \psi$ is still Anosov, such that $\dim E^u_f = 1,$ our Theorem \ref{t1} allows to conclude that $h_{top}(f) = h_{top}(A).$

The additional part of Theorem \ref{t1} is very easy in the  $C^2-$setting, by applying SRB theory of \cite{Bowen}, for instance. Here we deal with $C^1-$context and the SRB theory doesn't hold in general.

Particularly all algebraic linear  Anosov automorphism of torus or infranilmanifolds satisfies our hypothesis. In this way to find new examples of Anosov diffeomorphisms, modulo topological conjugacy, satisfying the hypothesis of additional part of Theorem \ref{t1} is as hard as to discovering new classes of Anosov diffeomorphisms, modulo topological conjugacy, (if any).

To finalize our comments we mention that M. Brin in \cite{B77} obtained sufficient conditions transitivity by a pinching condition on the uniform expanding and uniform contraction constants, in order to get a global product structure and so $\Omega(f) = M.$

Due to the scarcity of results  guaranteeing  the transitivity of Anosov diffeomorphisms, we believe that the techniques applied to prove Theorem \ref{t1} may be of special interest in the study of the transitivity of Anosov diffeomorphisms.

\section{Useful Technical Preliminaries }

To prove   Theorem \ref{t1} we will need some tools
related with uniform convergence of the limit which defines Lyapunov
exponents. This uniformity is related
with volume growth of the unstable foliation and consequently with the entropy along
the unstable leaves.

Let us recall results  given in \cite{AAS} and \cite{Cao}.

\begin{lemma}\label{lemmauniform2} Let $\mathcal{M}$ be the space of $f-$invariant measures, $\phi$ be a continuous function on $M.$ If $\alpha < \int \phi d\mu < \lambda, \; \forall \mu \in \mathcal{M}, $ then  there exists $N$ such that for all $n \geq N,$ we have
$$\alpha < \frac{1}{n} \sum_{i = 0}^{n - 1} \phi(f^i(x)) < \lambda,$$
 for all $x \in M.$
\end{lemma}

See  \cite{AAS} or \cite{Cao}  for the proofs of the above Lemma. The version $\int \phi d\mu > \alpha$ is not written in the aforementioned papers, but they take place in an analogous way.

An important tool for us is the notion of topological entropy along an expanding foliation. In \cite{Hua} the authors deal with a notion of topological entropy $h_{top}(f, \mathcal{W} )$ of an invariant expanding foliation $\mathcal{W}$ of a diffeomorphism $f. $ Among other results, they establish variational principle in this sense and a relation between $h_{top}(f, \mathcal{W} )$ and volume growth of $\mathcal{W}. $

Here $W(x)$  denotes the leaf of $\mathcal{W}$ by $x.$ Given  $\delta  > 0,$  we denote by $W(x, \delta)$ the $\delta-$ball centered in $x$ on $W(x),$ with the induced Riemannian distance, which is denoted by $d_{W}.$

Given $x \in M, $ $\varepsilon > 0, $ $\delta > 0$ and $n \geq 1$ a integer number, let $N_{W}(f, \varepsilon, n, x, \delta)$ be the maximal cardinality of a set $S \subset \overline{W(x, \delta)}$ such that $\displaystyle\max_{j =0 \ldots, n-1} d_{W}(f^j(z), f^j(y)) \geq \varepsilon,$ for any distinct points $y,z \in S.$

\begin{definition}\label{uentropy} The unstable entropy of $f$ on $M,$ with respect to the expanding foliation $\mathcal{W}$ is given by
$$h_{top}(f, \mathcal{W} ) = \lim_{\delta \rightarrow 0} \sup_{x \in M} h^{\mathcal{W}}_{top}(f, \overline{W(x, \delta)}), $$
where
$$h^{\mathcal{W}}_{top}(f, \overline{W(x, \delta)}) = \lim_{\varepsilon \rightarrow 0} \limsup_{n \rightarrow +\infty} \frac{1}{n} \log(N_{W}(f, \varepsilon, n, x, \delta)). $$
\end{definition}

Define $\mathcal{W}-$volume growth by
 $$\chi_{\mathcal{W}}(f) = \sup_{x \in M } \chi_{\mathcal{W}}(x, \delta), $$
where
$$ \chi_{\mathcal{W}}(x, \delta) = \limsup_{n\rightarrow +\infty} \frac{1}{n} \log(Vol(f^n(W(x, \delta)))).$$

Note that, since we are supposing $\mathcal{W}$ an expanding foliation, the above definition is independent of $\delta$ and the Riemannian metric.

\begin{theorem}[Theorem C and Corollary C.1 of \cite{Hua}]\label{teoH} With the above notations
$$h_{top}(f, \mathcal{W} ) = \chi_{\mathcal{W}}(f).$$
Moreover, $h_{top}(f) \geq h_{top}(f, \mathcal{W}). $
\end{theorem}

\section{Proof of Theorem \ref{t1}}

Consider $X$  a compact metric space and $f: X \rightarrow X$ a continuous map. The first fact that we observe is that if $f$ has a unique measure of maximal entropy, $\mu,$ then it is ergodic. This fact is a consequence of Jacobs Theorem.

\begin{lemma}[Jacobs Theorem] Suppose that $X$ is a completely separable metric
space and $f: X \rightarrow X$ a continuous function. Given any Borel invariant probability measure $\mu$ let $\{\mu_P: P \in \mathcal{P}\}$ be its
ergodic decomposition. Then, $$h_{\mu}(f) = \displaystyle\int_M h_{\mu_P}(f)d\hat{\mu}(P), $$ if one side is infinite,
so is the other side.
\end{lemma}
For a proof, we refer \cite{OV}.

When the topological entropy is finite, by Variational Principle, we observe that in the case where $\mu$ is the unique measure of maximal entropy, then, necessarily, it is an ergodic component.

The proof of Theorem \ref{t1} is a consequence of two lemmas.  Denote  $J^uf(x) = |\det Df(x): E^u_f(x)  \rightarrow E^u_f(f(x))| .$

\begin{lemma}\label{lem3-rafael}
Let $f: M \rightarrow M$ be a $C^1$ Anosov diffeomorphism and $\Lambda^u_f$ the common value $\Lambda^u_f(p), p \in Per(f).$ Then for any $x \in M ,$ holds $\Lambda^u_f  = \displaystyle\lim_{n \rightarrow +\infty} \frac{1}{n} \log(J^uf^n(x))$ and such limit is uniform on $M,$ particularly $h_{top}(f) = \Lambda^u_f.$
\end{lemma}

\begin{proof} Let $\mu$ be an $f-$invariant probability measure, and denote by $R$ the set of regular and recurrent points of $f,$ we have $\mu(R)  = 1.$

We can use  Anosov Closing Lemma to get
\begin{equation}
\Lambda^u_f(x):= \lim_{n \rightarrow + \infty} \frac{1}{n} \log(J^uf^n(x)) = \Lambda^u_f, \label{sameu}
\end{equation}
for any $x \in R.$ In fact, consider $\varepsilon > 0$ arbitrary. Since $f$ is  $C^1$ take $\delta > 0$ such that $1 - \varepsilon < \frac{J^uf(x)}{J^uf(y)} < 1 + \varepsilon,$ if $d(x,y) < \delta.$ Using that $x \in R,$ there is a sequence of integers $n_k, k =1,2,\ldots$ such that $d(x, f^{n_k}(x)) < \delta',$ for some $\delta' > 0$  such that every $\delta'-$pseudo orbit is indeed $\delta$ shadowed by a periodic orbit of a point $p_k$ such that $f^{n_k}(p_k) = p_k.$ So
$$ \frac{J^uf^{n_k}(x) }{J^uf^{n_k}(p_k)} = \frac{\prod_{j=0}^{n_k -1} J^uf (f^j(x))}{\prod_{j=0}^{n_k -1} J^uf (f^j(p_k))} \Rightarrow  (1 - \varepsilon)^{n_k} <\frac{J^uf^{n_k}(x) }{(\Lambda^u_f)^{n_k}} < (1+ \varepsilon)^{n_k}. $$
We get $ \Lambda^u_f + \log(1 - \varepsilon) \leq  \Lambda^u_f(x) = \displaystyle\lim_{k \rightarrow +\infty}\frac{1}{n_k} \log(J^uf^{n_k} (x)) \leq \Lambda^u_f + \log(1 + \varepsilon),$ so taking $\varepsilon $ going to zero, we obtain $\Lambda^u_f(x) = \Lambda^u_f,$ for any $x \in R.$

By Ruelle's inequality we have $h_{\mu}(f) \leq \Lambda^u_f,$ and using the Variational Principle, we get

\begin{equation}\label{topleq}
h_{top}(f) \leq \Lambda^u_f.
\end{equation}

 Since the identity $(\ref{sameu})$ holds on a full probability set,  using  Lemma \ref{lemmauniform2}  and the Birkhoff Ergodic Theorem with $\phi(x) = \log(J^uf(x)), x \in M,$ we conclude that the limit given in the expression \eqref{sameu} is uniform. So for any $\varepsilon > 0,$ there is $N > 0$ an integer number such that for any $n \geq N$ and $x \in M,$ we have
\begin{equation}
J^uf^n(x) > e^{n(\Lambda^u_f - \varepsilon)}.
\end{equation}
So given $B^u(x, \delta)$ a $u-$ball centered in $x$ with radius $\delta  >0,$ we have
\begin{equation}\label{growthu}
Vol_u(f^n(B^u(x, \delta))) = \int_{B^u(x, \delta)}  J^uf^n(x) dVol_u(x) > e^{n(\Lambda^u_f - \varepsilon)}Vol_u(B^u(x, \delta)),
\end{equation}
where $Vol_u$ denotes the $u-$dimension volume along unstable leaves induced by the Riemannian metric of $M.$

By  equation \eqref{growthu} we get $\chi_{\mathcal{F}^u}(f) \geq \Lambda^u_f - \varepsilon, $ for any $\varepsilon > 0.$ From Theorem \ref{teoH} we have $h_{top}(f, \mathcal{F}^u) \geq \Lambda^u_f,$ and
\begin{equation} \label{topgeq}
h_{top}(f)\geq \Lambda^u_f.
 \end{equation}

Now the equations \eqref{topleq} and \eqref{topgeq}, we get $h_{top}(f) = \Lambda^u_f. $

\end{proof}

\begin{lemma}\label{fudamental}
Let $f: M \rightarrow M$ be a $C^1-$ Anosov diffeomorphism such that $p \mapsto \Lambda^u_f(p)$ is constant on $Per(f).$ If $f$ has a unique measure of maximal entropy $\mu$ and it is absolutely continuous with respect to $m,$  such that the density $\rho =  \frac{d\mu}{dm}$ is continuous, then $\mu(U) > 0,$ for any non empty open set $U \subset M,$ consequently $f$ is transitive.
\end{lemma}

\begin{proof} As in the previous Lemma, consider $\Lambda^u_f$ the common value $\Lambda^u_f(p), p \in Per(f).$ Let $\rho$ be the continuous density of $\mu$ and consider the compact set $H_0 = \rho^{-1}(\{ 0\}),$ in particular $\mu(H_0)=0.$ Suppose that there exists a non empty open set $U \subset M,$ such that $\mu(U) = 0,$ so $U \subset H_0.$ In fact, if there was some $x \in U,$ such that $\rho(x) > 0,$ then by continuity of $\rho,$ we could choose an open set $V \subset U,$ such that $x \in V$ and $\rho(t) > \delta > 0,$ for some $\delta > 0$ and every $t \in V.$ With this $0< \mu(V) < \mu(U),$ that contradicts the fact $\mu(U) = 0.$

By invariance of $\mu,$ given $n \in \mathbb{Z},$ we have $f^n(U)$ is open set and $\mu(f^n(U)) = 0,$ so as before $f^n(U) \subset H_0.$ We conclude that $U \subset f^{-n}(H_0),$ for any $n \in \mathbb{Z}.$

Define $\Lambda = \displaystyle\bigcap_{n \in \mathbb{Z}} f^{-n}(H_0),$ we note that $U \subset \Lambda,$ and $\Lambda$ is a non empty and compact hyperbolic set invariant by $f.$

Since $U \subset \Lambda,$ we can take a local unstable arc $W$ of $W^u_f(x)$ fully contained in $U \subset \Lambda.$ Since $\Lambda$ contains an open unstable arc $W,$ proceeding as in Lemma \ref{lem3-rafael}   we have
$ h_{top}(f_{|\Lambda}, \mathcal{W}^u_{f_{|\Lambda}}) \geq \Lambda^u_f ,$ and by Theorem \ref{teoH} we get
$$ h_{top}(f_{|\Lambda}) \geq h_{top}(f_{|\Lambda}, \mathcal{W}^u_{f_{|\Lambda}}) \geq \Lambda^u_f = h_{top}(f) \geq h_{top}(f_{|\Lambda}).$$

So $h_{top}(f_{|\Lambda}) =  h_{top}(f),$ since $f_{|\Lambda}$ is expansive, then $f_{|\Lambda}$ admits a measure of maximal entropy $\nu,$ such that $\nu(\Lambda) = 1.$ Define $\overline{\nu}$ a Borel measure such that $\overline{\nu}(B) = \nu(B \cap \Lambda).$ Note that $\mu (\Lambda) = 0 = \overline{\nu}(M\setminus \Lambda).$  So $\mu$ and $\bar{\nu}$  are mutually singular measures of maximal entropy of $f.$ It is a contradiction with the assumption of $f$ having a unique measure of maximal entropy.

Since  $\mu$ is ergodic, for any non empty open sets $U , V \subset M,$ we have $\mu(U), \mu(V) > 0.$ By ergodicity, for $\mu$ a.e $x \in U,$
$$ \displaystyle\lim_{n \rightarrow +\infty}\frac{1}{n} \displaystyle\sum_{j=1}^{n-1} \chi_{V} (f^j(x)) = \mu(V) > 0,$$
so $f^j(x) \in V,$ for infinitely many $j > 0,$ thus  $f$ is transitive.

\end{proof}

\section{Further Questions}

\begin{question}
Let $f: M \rightarrow M$ be an Anosov diffeomorphism. If $p \mapsto \Lambda^u_f(p), p \in Per(f)$ is constant. Is $f$ transitive?
\end{question}

\begin{question}
Let $f: M \rightarrow M$ be an Anosov diffeomorphism. If $Jf^n(p) = 1$, for any  $p$ such that $f^n(p) = p,$ for some integer $n \geq 1,$ is $f$ transitive? Here $Jf(x)$ denotes the jacobian of $f$ at $x.$
\end{question}

\end{document}